\documentclass[final,leqno]{siamltex704}
\usepackage{amsmath}
\usepackage{graphicx}
\usepackage[notcite,notref]{showkeys}
\usepackage{mathrsfs}
\usepackage{float}
\usepackage{amsfonts,amssymb}
\usepackage{dsfont}
\usepackage{pifont}
\usepackage{hyperref}
\usepackage{multirow}
\numberwithin{equation}{section}
\def\3bar{{|\hspace{-.02in}|\hspace{-.02in}|}}
\def\E{{\mathcal{E}}}

\def\W{{\mathcal{W}}}
\def\V{{\mathcal{V}}}
\def\B{{\mathcal{B}}}
\def\btau{\boldsymbol{\tau}}

\def\bv{{\mathbf{v}}}

\def\blambda{{\boldsymbol{\lambda}}}
\def\jump#1{{[\![#1]\!]}}
\def\similarity#1{{\langle\!\langle#1\rangle\!\rangle}}

\newtheorem{remark}{Remark}[section]
\newtheorem{algorithm}{Weak Galerkin (WG) Algorithm}
\newtheorem{halgorithm}{Hybridized Weak Galerkin (HWG) Algorithm}
\newtheorem{variable-reduction}{Variable Reduction Algorithm}
\title{A Hybridized Weak Galerkin Finite Element Method for
the Biharmonic Equation}

\author{
Chunmei Wang\thanks{Nanjing Normal University Taizhou College,
Taizhou 225300; Jiangsu Key Laboratory for NSLSCS, School of
Mathematical Sciences, Nanjing Normal University, Nanjing 210023,
China.} \and Junping Wang\thanks{Division of Mathematical Sciences,
National Science Foundation, Arlington, VA 22230 (jwang@nsf.gov).
The research of Wang was supported by the National Science
Foundation IR/D program, while working at the Foundation. However,
any opinion, finding, and conclusions or recommendations expressed
in this material are those of the author and do not necessarily
reflect the views of the National Science Foundation.}}

\begin{document}

\maketitle

\begin{abstract}
This paper presents a hybridized formulation for the weak Galerkin
finite element method for the biharmonic equation. The hybridized
weak Galerkin scheme is based on the use of a Lagrange multiplier
defined on the element boundaries. The Lagrange multiplier is
verified to provide a numerical approximation for certain
derivatives of the exact solution. An optimal order error estimate
is established for the numerical approximations arising from the
hybridized weak Galerkin finite element method. The paper also
derives a computational algorithm (Schur complement) by eliminating
all the unknown variables on each element, yielding a significantly
reduced system of linear equations for unknowns on the boundary of
each element.
\end{abstract}



\pagestyle{myheadings}

\section{Introduction}
In this paper, we are concerned with new developments of weak
Galerkin finite element methods for partial differential equations.
In particular, we shall employ the usual hybridization technique
\cite{fv1965,ab1985,cgl2009} to the weak Galerkin finite element
method for the biharmonic equations proposed and analyzed in
\cite{ww}.

For simplicity, we consider the following biharmonic equation with
Dirichlet and Neumann boundary conditions:
 \begin{equation}\label{0.1}
 \begin{split}
 \Delta^2u&=f, \quad\text{in}\ \Omega,\\
 u&=\xi, \quad\text{on}\ \partial\Omega,\\
 \frac{\partial u}{\partial\textbf{n}}&=\nu, \quad
 \text{on}\ \partial\Omega,\\
 \end{split}
 \end{equation}
where $\Omega$ is an open bounded domain in the Euclidean space
$\mathbb{R}^d\ (d=2, 3)$ with Lipschitz continuous boundary
$\partial\Omega$.

The weak Galerkin method is a finite element technique that
approximates differential operators (e.g., gradient, divergence,
curl, Laplacian, Hessian, etc) as distributions. The method has been
successfully applied to several classes of partial differential
equations, such as the second order elliptic equation
\cite{wy2013,wy3655,wy}, the Stokes equation \cite{wang-ye-stokes},
the Maxwell's equations \cite{mwyz-maxwell}, and the biharmonic
equation \cite{mwy, ww}. For example, in \cite{ww}, a weak Galerkin
finite element method was developed for the biharmonic equation
(\ref{0.1}) by using polynomials of degree $P_k/P_{k-2}/P_{k-2}$ for
any $k\ge 2$, where $P_k$ was used to approximate the function $u$
on each element and $P_{k-2}$ was employed to approximate the trace
of $u$ and $\nabla u$ on the element boundary. The objective of this
paper is to exploit the use of hybridization techniques in weak
Galerkin methods that shall further relax the connection of the
finite element functions among elements.

Hybridization is a useful technique in finite element methods. The
key to hybridization is to identify a Lagrange multiplier which can
be used to relax certain constrains (e.g., continuity) imposed on
the finite element function across element boundaries. Hybridization
has been employed in mixed finite element methods to yield
hybridized mixed finite element formulations suitable for efficient
implementation in practical computation
\cite{ab1985,b1973,b1974,bdm1985,fv1965,rt1977,w}. The idea of
hybridization was also used in discontinuous Galerkin methods
\cite{abcm2002} for deriving hybridized discontinuous Galerkin (HDG)
finite element methods \cite{cgl2009}.

We shall show in this paper that hybridization is a natural approach
for weak Galerkin finite element methods. For illustrative purpose,
we demonstrate how hybridization can be accomplished for the weak
Galerkin finite element scheme of \cite{ww}. We shall also establish
a theoretical foundation to address critical issues such as
stability and convergence for the hybridized weak Galerkin (HWG)
finite element method. The hybridized weak Galerkin is further used
as a tool to derive a Schur complement problem for variables defined
on element boundaries. Therefore, the Schur complement involves the
solution of a linear system with significantly less number of
unknowns than the original WG or HWG formulation. We believe the
hybridization technique is widely applicable in weak Galerkin family
for various partial differential equations, and would like to
encourage interested readers to conduct some independent study along
this direction.

The paper is organized as follows. In Section \ref{Section:Hessian},
we introduce a weak Hessian and a discrete weak Hessian by using
polynomial approximations. In Section \ref{Section:HWG}, we present
a HWG finite element algorithm for the biharmonic problem
(\ref{0.1}). In Section \ref{Section:Stability}, we verify all the
stability conditions in Brezzi's theorem \cite{b1973} for the HWG
scheme. In Section \ref{Section:Error-Eqns}, we derive an error
equation for the HWG approximation. In Section
\ref{Section:Error-Estimates}, we establish an optimal-order error
estimate for the numerical approximation. Finally in Section
\ref{Section:Variable-Reduction}, we present a Schur complement by
eliminating all the variables on the element, yielding a system of
linear equations with significantly reduced number of unknowns
defined on the element boundary.

\section{Weak Hessian and Discrete Weak
Hessian}\label{Section:Hessian} Let $T$ be a polygonal or polyhedral
domain with boundary $\partial T$. A weak function on $T$ is one
given by $v=\{v_0,v_b,\textbf{v}_g\}$ such that $v_0\in L^2(T)$,
$v_b\in L^{2}(\partial T)$ and $\textbf{v}_g\in [L^{2}(\partial
T)]^d$. Let $\W(T)$ be the space of all weak functions on $T$; i.e.,
\begin{equation}\label{2.1}
\W(T)=\{v=\{v_0,v_b,\textbf{v}_g\}: v_0\in L^2(T), v_b\in
L^{2}(\partial T), \textbf{v}_g\in [L^{2}(\partial T)]^d\}.
\end{equation}

Throughout the paper, $C$ appearing in different places denotes
different constant. $(\cdot,\cdot)_{T}$ and
$\langle\cdot,\cdot\rangle_{\partial T}$ denote the usual inner
products in $L^2(T)$ and $L^2(\partial T)$. Denote by
$\|\cdot\|_{m,K}$ the norm in the Sobolev space $H^m(K)$. Let
$|\cdot|_{m,\Omega}$ be the semi-norm of order $m$. For simplicity,
$\|\cdot\|_{m,\Omega}$, $|\cdot|_{m,\Omega}$,
$(\cdot,\cdot)_{\Omega}$ and $\langle\cdot,\cdot\rangle_{\Omega}$
are denoted as $\|\cdot\|_{m}$, $|\cdot|_{m}$, $(\cdot,\cdot)$ and
$\langle\cdot,\cdot\rangle$, respectively. $\|\cdot\|_{0,T}$,
$\|\cdot\|_{0,\partial T}$ and $|\cdot|_{0,\partial \Omega}$ are
simply denoted by $\|\cdot\|_{T}$, $\|\cdot\|_{\partial T}$ and
$\|\cdot\|_{\partial \Omega}$, respectively.

For classical functions, the Hessian is a square matrix of second
order partial derivatives if they all exist. If $f(x_1, \cdots,
x_d)$ stands for the function, then the Hessian of $f$ is
$$
H(f) = \left\{\partial_{ij}^2 f \right\}_{d\times d},
$$
where $\partial_{ij}^2$ is the second order partial derivative along
the directions $x_i$ and $x_j$. The goal of this section is to
introduce weak Hessian for weak functions defined on $T$.

For any $v\in \W(T)$, the weak partial derivative $\partial^2_{ij}$
of $v$ is defined as a linear functional $\partial^2_{ij,w} v$ in
the dual space of $H^2(T)$ such that
 \begin{equation}\label{2.3}
 (\partial^2_{ij,w}v,\varphi)_T=(v_0,\partial^2_{ji}\varphi)_T-
 \langle v_b n_i,\partial_j\varphi\rangle_{\partial T}+
 \langle v_{gi},\varphi n_j\rangle_{\partial T}
 \end{equation}
for all $\varphi\in H^2(T)$. Here $\textbf{n}=(n_1,\cdots,n_d)$ is
the outward normal direction of $T$ on its boundary. The weak
Hessian is then defined as
$$
H_{w,T}(v) = \left\{\partial_{ij,w}^2 v \right\}_{d\times d},\qquad
v\in \W(T).
$$

A discrete version of $\partial^2_{ij,w}$ is an approximation,
denoted by $\partial^2_{ij,w,r,T}$, in the space of polynomials of
degree $r$ such that
  \begin{equation}\label{2.4}
 (\partial^2_{ij,w,r,T}v,\varphi)_T=(v_0,\partial^2
 _{ji}\varphi)_T-\langle v_b n_i,\partial_j\varphi\rangle_{\partial T}
 +\langle v_{gi},\varphi n_j\rangle_{\partial T},\quad \forall \varphi \in
 P_r(T).
 \end{equation}
Analogously, the discrete Hessian is given by
$$
H_{w,r,T}(v) = \left\{\partial_{ij,w,r,T}^2 v \right\}_{d\times
d},\qquad v\in \W(T).
$$

\begin{remark}
Let $v=\{v_0,v_b,\textbf{v}_g\}\in \W(T)$ be a weak function on $T$
such that $v_0$ is twice differentiable on $T$. By applying the
usual integration by parts to the first term on the right-hand side
of (\ref{2.4}), we obtain
\begin{equation}\label{2.4new}
 (\partial^2_{ij,w,r,T}v,\varphi)_T=(\partial^2_{ij}v_0,\varphi)_T
+\langle (v_0-v_b)n_i,\partial_j\varphi \rangle_{\partial T}
 -\langle (\partial_i v_0 - v_{gi})n_j,\varphi\rangle_{\partial T}
 \end{equation}
 for all $\varphi \in P_r(T)$.
\end{remark}

\section{A Hybridized Weak Galerkin Formulation}\label{Section:HWG}
The goal of this section is to introduce a hybridized formulation
for the weak Galerkin finite element algorithm that was first
designed in \cite{ww}.

\subsection{Notations}
Let ${\cal T}_h$ be a partition of the domain $\Omega$ into polygons
in 2D or polyhedra in 3D. Denote by ${\mathcal E}_h$ the set of all
edges or flat faces in ${\cal T}_h$ and  ${\mathcal E}_h^0={\mathcal
E}_h \setminus
\partial\Omega$ the set of all interior edges or flat faces. Assume
that ${\cal T}_h$ is shape regular as described in \cite{wy}. Denote
by $h_T$ the diameter of $T\in {\cal T}_h$ and $h=\max_{T\in {\cal
T}_h}h_T$ the meshsize for the partition ${\cal T}_h$.

For each element $T\in {\cal T}_h$, the trace of $\W(T)$ on the
boudary $\partial T$ is the usual Sobolev space $L^2(\partial
T)\times [L^2(\partial T)]^d$. Define the spaces $\W$ and $\Lambda$
by
\begin{equation}\label{3.1}
\W=\prod_{T\in{\cal T}_h}\W(T),\qquad \Lambda=\prod_{T\in{\cal T}_h}
L^2(\partial T)\times [L^2(\partial T)]^d.
\end{equation}
It should be pointed out that the values of functions in the space
$\W$ are not correlated between any two adjacent elements $T_1$ and
$T_2$ which share $e\subset {\mathcal E}_h^0$ as a common edge or
flat face. For example, on each interior edge $e\subset {\mathcal
E}_h^0$, $v\in\W$ has two copies of $v_b$; one taken from the left
(say $T_1$) and the other from the right (say $T_2$). Similarly, the
vector component $\bv_g$ has two values: left from $T_1$ and right
from $T_2$. Define the jump of $v\in \W$ on $e\subset {\mathcal
E}_h$ by
\begin{equation}\label{Eq:jump}
\begin{split}
\jump{v}_e=\left\{\begin{array}{ll}
                 \{v_b,\textbf{v}_g\}|_{\partial T_1}-\{v_b,\textbf{v}_g\}|_{\partial T_2}, &
                 e\in {\mathcal E}_h^0, \\
                 \{v_b,\textbf{v}_g\}, & e\subset \partial\Omega,
               \end{array}\right.
\end{split}
\end{equation}
where $\{v_b,\textbf{v}_g\}|_{\partial  T_i}$ denotes the value of
$\{v_b,\textbf{v}_g\}$ on $e$ as seen from the element $T_i,\
i=1,2$. The order of $T_1$  and $T_2$ is non-essential in
(\ref{Eq:jump}) as long as the difference is taken in a consistent
way in all the formulas. We shall also use the notation
$\{v_b,\textbf{v}_g\}_{L}$ for $\{v_b,\textbf{v}_g\}|_{\partial
T_1}$ and $\{v_b,\textbf{v}_g\}_{R}$ for
$\{v_b,\textbf{v}_g\}|_{\partial T_2}$ in the rest of the paper.

For any function $\lambda \in \Lambda$, define its similarity on
$e\subset {\mathcal E}_h$ by
\begin{equation}\label{Eq:similarity}
\begin{split}
\similarity{\lambda}_e=\left\{\begin{array}{ll}
\{\lambda_b,\boldsymbol{\lambda}_g\}_{L}+\{\lambda_b,\boldsymbol{\lambda}_g\}_{R},
&
e\in{\mathcal E}_h^0, \\
\{\lambda_b,\boldsymbol{\lambda}_g\}, & e\subset \partial\Omega.
\end{array}\right.
\end{split}
\end{equation}
Denote by $\similarity{\lambda}$ the similarity of $\lambda$ in
${\mathcal E}_h$.

For any given integer $k\geq 2$, denote by $\W_k(T)$ the discrete
weak function space given by
\begin{equation*}
\W_k(T)=\big\{\{v_0,v_b,\textbf{v}_g\}: v_0\in P_k(T), v_b\in
P_{k-2}(e),\textbf{v}_g\in [P_{k-2}(e)]^d, e\subset \partial
T\big\}.
\end{equation*}
Denote by $\Lambda_k(\partial T)$ the trace of $\W_k(T)$ on the
boundary $\partial T$; i.e.,
\begin{equation}
\Lambda_k(\partial
T)=\{\lambda=\{\lambda_b,\boldsymbol{\lambda}_g\}:\lambda_b|_e \in
P_{k-2}(e), \boldsymbol{\lambda}_g|_e\in [P_{k-2}(e)]^d,e\subset
\partial T\}.
\end{equation}
By patching $\W_k(T)$ and $\Lambda_k(\partial T)$ over all the
elements $T\in {\cal T}_h$, we obtain two weak Galerkin finite
element spaces $\W_h$ and $\Lambda_h$ as follows
\begin{equation}
\W_h=\prod_{T\in{\cal T}_h}\W_k(T),\qquad \Lambda_h=\prod_{T\in{\cal
T}_h} \Lambda_k(\partial T).
\end{equation}

Denote by $\W_h^0$ the subspace of $\W_h$ consisting of functions
with vanishing boundary values
$$
\W_h^0=\{v\in \W_h:\ \ v_b|_e=0,\textbf{v}_g|_e=\textbf{0},
e\subset\partial\Omega\}.
$$
Furthermore, let $\V_h$ be the subspace of $\W_h$ consisting of
functions which are continuous across each interior edge or flat
face
$$
\V_h=\big\{v\in \W_h:\ \ \jump{v}_e=\{0,\textbf{0}\},  e\in
{\mathcal E}_h^0\big\}.
$$
Denote by $\V_h^0$ a subspace of $\V_h$ consisting of functions with
vanishing boundary values
$$
\V_h^0=\{v\in \V_h: \ \ v_b|_e=0,\textbf{v}_g|_e=\textbf{0},
e\subset\partial\Omega\}.
$$

Let $\Xi_h$ be the subspace of $\Lambda_h$ consisting of functions
with similarity zero across each edge or flat face; i.e.,
\begin{equation*}
\begin{split}
\Xi_h =\Big\{\lambda\in \Lambda_h:
\similarity{\lambda}_e=\{0,\textbf{0}\}, e\in {\mathcal E}_h\Big\}.
  \end{split}
\end{equation*}
The functions in the space $\Xi_h$ serve as Lagrange multipliers in
hybridization methods.

Denote by $H_{w,k-2}$ the discrete weak Hessian in the finite
element space $\V_h$, which is computed by using (\ref{2.4}) on each
element $T$ by
$$
(\partial^2_{ij, w,k-2} v)|_T=\partial^2_{ij,w,k-2,T}(v|_T), \qquad
v\in \V_h.
$$
For simplicity of notation, we shall drop the subscript $k-2$ from
the notation $\partial^2_{ij, w,k-2}$ and $H_{w,k-2}$ in the rest of
the paper. We also introduce the following notation
$$
(\partial^2_{w}u,\partial^2_{w}v)_h=\sum_{T\in{\cal
T}_h}\sum_{i,j=1}^d (\partial^2_{ij,w}u,\partial^2_{ij,w}v)_T,\quad
\forall u, v\in \V_h.
$$

On each element $T$, denote by $Q_0$ the $L^2$ projection onto
$P_k(T)$. Similarly, for each edge or face $e\subset\partial T$,
denote by $Q_b$ the $L^2$ projection onto $P_{k-2}(e)$ or
$[P_{k-2}(e)]^d$, as appropriate. For any $q\in H^2(\Omega)$, define
a projection $Q_h q$ onto the weak finite element space $\V_h$ such
that on each element $T$
$$
Q_h q=\{Q_0 q,Q_b q,Q_b(\nabla q)\}.
$$

\subsection{Algorithm}
For any $w=\{w_0,w_b,\textbf{w}_g\}\in \W_k(T)$ and
$v=\{v_0,v_b,\textbf{v}_g\}\in \W_k(T)$ and $\lambda\in
\Lambda_k(\partial T)$, set
\begin{align*}
a_T(w,v)=&\sum_{i,j=1}^d(\partial^2_{ij,w}w,\partial^2_{ij,w}v)_T,
\\
s_T(w,v)=& h_T^{-1}\langle Q_b(\nabla w_0)
-\textbf{w}_g, Q_b(\nabla v_0)-\textbf{v}_g\rangle_{\partial T} \\
&+h_T^{-3}\langle Q_bw_0-w_b,Q_bv_0-v_b\rangle_{\partial T},
\\
b_T(v,\lambda)=&\langle v,\lambda\rangle_{\partial T}\\
=&\langle v_b,\lambda_b\rangle_{\partial T}+\langle
\textbf{v}_g,\boldsymbol{\lambda}_g\rangle_{\partial T}.
\end{align*}
Define
$$
a_{s,T}(w,v)=a_T(w,v)+s_T(w,v).
$$
Summing over all the elements $T\in {\cal T}_h$ yields four bilinear
forms
\begin{align*}
a(w,v)&=\sum_{T\in {\cal T}_h}a_T(w,v),\qquad w, v\in \W_h,
\\
s(w,v)&=\sum_{T\in {\cal T}_h}s_T(w,v),\qquad w, v\in \W_h,
\\
b(v,\lambda)&=\sum_{T\in {\cal T}_h}b_T(v,\lambda),\qquad
 v\in \W_h, \lambda\in \Lambda_h,
\\
a_s(w,v)&=\sum_{T\in {\cal T}_h}a_{s,T}(w,v),\qquad w, v\in \W_h.
\end{align*}
Since $\lambda\in\Xi_h$ implies $\lambda_L+\lambda_R=0$ on each
interior edge and $\lambda=0$ on the boundary edge, then for any
$v\in\W_h$ and $\lambda\in\Xi_h$, we have
\begin{equation}\label{EQ:b-form}
b(v,\lambda) = \sum_{e\in\E_h^0} \langle \jump{v}_e,
\lambda_L\rangle_e.
\end{equation}

The following weak Galerkin finite element scheme for the biharmonic
equation (\ref{0.1}) was introduced and analyzed in \cite{ww}.

\begin{algorithm}\label{algo1} Find $\bar{u}_h=\{\bar{u}_0,\bar{u}_b,
\bar{\textbf{u}}_g\}\in \V_h$ such that $\bar{u}_b=Q_b\xi$,
$\bar{\textbf{u}}_g\cdot \textbf{n}=Q_{b}\nu$,
$\bar{\textbf{u}}_g\cdot
\boldsymbol{\tau}=Q_{b}(\nabla\xi\cdot\btau)$ on $\partial\Omega$
and satisfying
\begin{equation}\label{algori1}
a_s(\bar{u}_h,v)=(f,v_0), \quad \forall
v=\{v_0,v_b,\textbf{v}_g\}\in \V_h^0,
\end{equation}
where $\boldsymbol{\tau}\in \mathbb{R}^d$ is the tangential
direction to the edges or faces on the boundary of $\Omega$.
\end{algorithm}

Next, we proposed a hybridized formulation for (\ref{algori1}) by
using a Lagrange multiplier.

\begin{halgorithm} \label{algo2}
Find $(u_h;\lambda_h)\in \W_h \times \Xi_h$ such that
${u}_b=Q_b\xi$, ${\textbf{u}}_g\cdot \textbf{n}=Q_{b}\nu$,
${\textbf{u}}_g\cdot \boldsymbol{\tau}=Q_{b}(\nabla\xi\cdot\btau)$
on $\partial\Omega$ and satisfying the following equations
\begin{eqnarray}\label{2.7}
a_s(u_h,v)-b(v, \lambda_h)&=&(f,v_0), \quad\forall v
=\{v_0,v_b,\textbf{v}_g\}\in \W_h^0,\\
b(u_h,\rho)&=&0,\quad\qquad\forall \rho \in\Xi_h.\label{2.7q2}
\end{eqnarray}
\end{halgorithm}

\subsection{The Relation between WG and HWG}

The HWG scheme (\ref{2.7})-(\ref{2.7q2}) is in fact equivalent to
the WG scheme (\ref{algori1}) in that the solution $u_h$ from
(\ref{2.7})-(\ref{2.7q2}) and $\overline{u}_h$ from (\ref{algori1})
are identical. But the HWG scheme (\ref{2.7})-(\ref{2.7q2}) is
expected to be advantageous over WG for biharmonic interface
problems.

For any $v\in\V_h^0$, let
\begin{equation}\label{lemnorm}
\3bar v\3bar =a^{\frac{1}{2}}_s(v,v).
\end{equation}
It has been verified in \cite{ww} that (\ref{lemnorm}) defines a
norm in the linear space $\V_{h}^0$.

\begin{theorem}\label{lem3.1}
Let $u_h\in \W_h$ be the first component of the solution of the
hybridized WG algorithm (\ref{2.7})-(\ref{2.7q2}). Then, we have
$\jump{u_h}_e=0$ on each interior edge or flat face $e\in {\mathcal
E}_h^0$; i.e., $u_h\in \V_h$. Furthermore, we have $u_b=Q_b\xi$,
$\textbf{u}_g\cdot \textbf{n}=Q_{b}\nu$, $\textbf{u}_g\cdot
\boldsymbol{\tau}=Q_{b}(\nabla\xi\cdot\btau)$ on $\partial\Omega$
and $u_h$ satisfies the equation (\ref{algori1}). Thus, one has $u_h
= \bar u_h$.
\end{theorem}

\begin{proof} Let $e$ be an interior edge or flat face
shared by two elements $T_1$ and $T_2$. By letting
$\rho=\jump{u_h}_e$ on $e$ as seen from  $T_1$ (i.e.,
$\rho=-\jump{u_h}_e$ on $e$ as seen from $T_2$) and $\rho=0$
otherwise in (\ref{2.7q2}), we obtain from (\ref{EQ:b-form}) that
$$
0=b(u_h,\rho)=\sum_{T\in {\cal T}_h}\langle
u_h,\rho\rangle_{\partial T}=\int_e\jump{u_h}_e^2ds,
$$
which implies that $\jump{u_h}_e=0$ for each interior edge or flat
face $e\in {\mathcal E}_h^0$.

Now by restricting $v\in \V_h^0$ in the equation (\ref{2.7}) and
using the fact that $b(v, \lambda_h)=0$, we arrive at
$$
a_s(u_h,v)=(f,v_0)\qquad \forall v\in \V_h^0,
$$
which is the same as (\ref{algori1}). It follows from the solution
uniqueness for (\ref{algori1}) that $u_h\equiv \bar u_h$. This
completes the proof.
\end{proof}

\section{Stability Conditions for HWG}\label{Section:Stability}

It is easy to see that the following defines a norm in the finite
element space $\Xi_h$
\begin{equation}\label{xinorm}
\|\lambda_h\|_{\Xi_h}=\Big(\sum_{e\in {\mathcal E}_h^0}
 h_e^{3}\|\lambda_b
 \|^2_e+h_e \|\boldsymbol{\lambda}_g\|^2_e\Big)^{\frac{1}{2}}.
 \end{equation}
As to $\W_h^0$, for any $v=\{v_0,v_b,\textbf{v}_g\}\in \W_h^0$, let
\begin{equation}\label{wh0norm}
\|v\|_{\W_h^0} =\Big(\3bar v\3bar  ^2+\sum_{e\in {\mathcal
E}_h^0}h_e^{-3}\|\jump{v_b}_e\|_e^2+h_e^{-1}\|\jump{\textbf{v}_g}_e
\|_e^2\Big)^{\frac{1}{2}}.
\end{equation}
We claim that $\|\cdot\|_{\W_h^0}$ defines a norm in $\W_h^0$. In
fact, if $\|v\|_{\W_h^0}=0$, then $\jump{v_b}_e=0$ and
$\jump{\textbf{v}_g}_e=\textbf{0}$ on each interior edge or flat
face $e\in {\mathcal E}_h^0$, and hence $v\in \V_h^0$. Since
$\3bar\cdot\3bar$ defines a norm in the linear space $\V_h^0$, then
$v=0$. This verifies the positivity property of
$\|\cdot\|_{\W_h^0}$. The other properties for a norm can be checked
trivially.

\begin{lemma}{\rm(\cite{wy})}\label{Lemma:trace inequality}
~\emph{\rm (}Trace Inequality{\rm)} Let $\mathcal{T}_h$ be a
partition of the domain $\Omega$ into polygons in 2D or polyhedra in
3D. Assume that the partition $\mathcal{T}_h$ satisfies the
assumptions (P1), (P2), and (P3) as specified in \cite{wy}. Let
$p>1$ be any real number. Then, there exists a constant $C$ such
that for any $T\in \mathcal{T}_h$ and edge/face $e\in\partial T$, we
have
\begin{eqnarray}\label{eq:trace inequality}\|\theta\|^p_{L^p(e)}\leq
Ch_T^{-1}(\|\theta\|^p_{L^p(T)}+h^p_T\|\nabla\theta\|^p_{L^p(T)}),
\end{eqnarray}
where $\theta\in W^{1,p}(T)$ is any function.
\end{lemma}

This paper will make use of the trace inequality (\ref{eq:trace
inequality}) with $p=2$:
\begin{eqnarray}\label{eq:trace inequality-p=2}\|\theta\|^2_e\leq
Ch_T^{-1}\|\theta\|^2_{T}+Ch_T\|\nabla\theta\|^2_{T}.
\end{eqnarray}

\begin{lemma}\label{lemboundedness}(boundedness)
There exists a constant $C>0$ such that
\begin{align}\label{bou1}
|a_s(u ,v )|&\leq C \|u \|_{\W_h^0}\|v \|_{\W_h^0}, \quad  \forall
u, v \in \W_h^0,\\
 |b(v ,\lambda )|&\leq C
\|v \|_{\W_h^0}\|\lambda \|_{\Xi_h}, \quad \forall v \in \W_h^0,
\lambda \in \Xi_h.\label{bou2}
\end{align}
\end{lemma}

\begin{proof}
To prove (\ref{bou1}), we use the Cauchy-Schwarz inequality to
obtain
\begin{align*}
|a_s(u,v)|=&\Big|\sum_{T\in{\cal
T}_h}\sum_{i,j=1}^d(\partial_{ij,w}^2u,\partial_{ij,w}^2v)_T+h_T^{-1}
\langle Q_b(\nabla u_0)-\textbf{u}_g,Q_b(\nabla
v_0)-\textbf{v}_g\rangle_{\partial T}\\
&+h_T^{-3} \langle Q_bu_0- u_b,Q_b v_0 - v_b \rangle_{\partial
T}\Big|\\
\leq&\Big(\sum_{T\in{\cal
T}_h}\sum_{i,j=1}^d\|\partial_{ij,w}^2u\|^2_T\Big)^\frac{1}{2}
\Big(\sum_{T\in{\cal
T}_h}\sum_{i,j=1}^d\|\partial_{ij,w}^2v\|^2_T\Big) ^\frac{1}{2}\\
&+\Big(\sum_{T\in{\cal T}_h} h_T^{-1} \| Q_b(\nabla
u_0)-\textbf{u}_g\|_{\partial T}^2\Big)
^\frac{1}{2}\Big(\sum_{T\in{\cal T}_h} h_T^{-1} \| Q_b(\nabla
v_0)-\textbf{v}_g\|_{\partial T}^2\Big) ^\frac{1}{2}\\
&+\Big(\sum_{T\in{\cal T}_h} h_T^{-3} \| Q_b u_0 - u_b\|_{\partial
T}^2\Big) ^\frac{1}{2}\Big(\sum_{T\in{\cal T}_h} h_T^{-3} \| Q_b
v_0 - v_b\|_{\partial T}^2\Big) ^\frac{1}{2}\\
\leq& C\|u\|_{\W_h^0}\|v\|_{\W_h^0}.
\end{align*}

As to (\ref{bou2}), it follows from the Cauchy-Schwarz inequality
that
\begin{align*}
 |b(v,\lambda)| =&\Big|\sum_{T\in {\cal T}_h}\langle
v_b,\lambda_b\rangle_{\partial T}+\langle
\textbf{v}_g,\boldsymbol{\lambda}_g\rangle_{\partial T}\Big|\\
=&\Big|\sum_{e\in {\cal E}_h^0}\langle
\jump{v_b},\lambda_b\rangle_{e}+\langle
\jump{\textbf{v}_g},\boldsymbol{\lambda}_g\rangle_{e}\Big|\\
 \leq & \Big(\sum_{e\in {\cal E}_h^0}h_e^{-3}
\|\jump{v_b}\|^2_{e}\Big)^{\frac{1}{2}}\Big(\sum_{e\in {\cal E}_h^0}
h_e^{3}\|\lambda_b\|^2_{e}\Big)^{\frac{1}{2}}\\
&+\Big(\sum_{e\in {\cal E}_h^0} h_e
^{-1}\|\jump{\textbf{v}_g}\|^2_{e}\Big)^{\frac{1}{2}}\Big(\sum_{e\in
{\cal E}_h^0}
h_e\|\boldsymbol{\lambda}_g\|^2_{e}\Big)^{\frac{1}{2}}\\
 \leq & C \|v\|_{\W_h^0}\|\lambda \|_{\Xi_h},
\end{align*}
which ends the proof.
\end{proof}

\begin{lemma}\label{lemcovercivity}(coercivity)
There exists a constant $C>0$, such that
\begin{equation}\label{q}
a_s(v,v)\geq C \|v\|_{\W_h^0}^2,\qquad \forall v\in \V_{h}^0.
\end{equation}
\end{lemma}

\begin{proof} For any $v\in \V_{h}^0$, we have
$\|v\|_{\W_h^0}=\3bar v\3bar.$ Thus, the estimate (\ref{q}) holds
true with $C=1$.
\end{proof}

\begin{lemma}\label{lemmainfsup}(inf-sup condition)
 There exists a constant $C>0$ such that
\begin{equation}\label{a3}
\sup_{v\in \W_h^0}\frac{b(v,\sigma)}{\|v\|_{\W_h^0}}\geq C
\|\sigma\|_{\Xi_h}, \qquad \forall \sigma\in \Xi_h.
\end{equation}
\end{lemma}

\begin{proof} For any $\sigma\in\Xi_h$, we have
$\similarity{\sigma}_e=0$ or equivalently $\sigma^L+\sigma^R=0$ on
each interior edge $e\in \E_h^0$ and $\sigma=0$ on all boundary
edges. By letting $v=\{0,h_e^3\sigma_b,h_e\boldsymbol{\sigma}_g\}\in
\W_h^0$ in $b(v,\sigma)$ and $s(v,v)$, we obtain
\begin{equation}\label{a5}
\begin{split}
b(v,\sigma)&=\sum_{e\in {\mathcal E}_h^0}\langle
v_b^L,\sigma_b^L\rangle_e+\langle v_b^R,\sigma_b^R\rangle_e+ \langle
\textbf{v}_g^L ,\boldsymbol{\sigma}_g^L\rangle_e+ \langle
\textbf{v}_g^R,\boldsymbol{\sigma}_g^R\rangle_e
\\
&=\sum_{e\in {\mathcal E}_h^0}\langle
v_b^L-v_b^R,\sigma_b^L\rangle_e+ \langle
\textbf{v}_g^L-\textbf{v}_g^R,\boldsymbol{\sigma}_g^L\rangle_e
\\
& =2\sum_{e\in {\mathcal E}_h^0} h_e^3\| \sigma_b\|^2_e+h_e\|
\boldsymbol{\sigma}_g\|^2_e,
\end{split}
\end{equation}
and
\begin{equation}\label{a6}
\begin{split}
s(v,v)=&\sum_{e\in {\mathcal E}_h^0} h_e^{-1}h_e^{2}\|
 \boldsymbol{\sigma}_g^L\|^2_{e}
+h_e^{-3}h_e^{6}\|  \sigma_b^L\|^2_e\\
&+h_e^{-1}h_e^{2}\| \boldsymbol{\sigma}_g^R\|^2_{e}
+h_e^{-3}h_e^{6} \|  \sigma_b^R\|^2_{e}\\
=&2\sum_{e\in {\mathcal E}_h^0} h_e \|
 \boldsymbol{\sigma}_g\|^2_{e}
+h_e^{3}\|  \sigma_b\|^2_e.
\end{split}
\end{equation}
It follows from (\ref{2.4}), Cauchy-Schwarz inequality, the trace
inequality (\ref{eq:trace inequality-p=2}) and the inverse
inequality that
\begin{equation}\label{a7}
\begin{split}
&(\partial_{ij,w}^2v,\partial_{ij,w}^2v)_T \\
=&\sum_{e\subset
\partial T} -\langle v_b^*,\partial_j(\partial_{ij,w}^2v)\cdot
n_i\rangle_e+\langle v_{gi}^*\cdot n_j,
\partial_{ij,w}^2v\rangle_e \\
\leq & \sum_{e\subset
\partial T} h_e^3\|\sigma_b^*\|_e\|\partial_j(\partial_{ij,w}^2v)\|_e+h_e\| \sigma_{gi}^*\|_e
\|\partial_{ij,w}^2v\|_e \\
\leq & C\sum_{e\subset
\partial T}h_e^3 \|\sigma_b^*\|_eh_e^{-3/2}\| \partial_{ij,w}^2v\|_T+h_e\| \sigma_{gi}^*\|_eh_e^{-1/2}
\|\partial_{ij,w}^2v\|_T \\
=& C\sum_{e\subset
\partial T}\| \partial_{ij,w}^2v\|_T\Big(h_e^{\frac{3}{2}}\|\sigma_b^*\|_e+h_e^{\frac{1}{2}}\| \sigma_{gi}^*\|_e
 \Big),
\end{split}
\end{equation}
where $v_b^*$ is chosen to be $v_b^L$ or $v_b^R$ according to the
relative position of $v_b$ and $e$, and the same to $v_{gi}^*$,
$\sigma_{b}^*$, $ \sigma_{gi}^*$, which implies that
\begin{equation}\label{a8}
\| \partial_{ij,w}^2v\|_T \leq C\sum_{e\subset
\partial T} h_e^{\frac{3}{2}}\|\sigma_b^*\|_e+h_e^{\frac{1}{2}}\| \sigma_{gi}^*\|_e.
\end{equation}
Summing over all element $T$ yields
\begin{equation}\label{a9}
\begin{split}
(\partial^2_w v, \partial^2_w v)_h\leq C\sum_{e\in{\mathcal
E}_h^0}\sum_{i=1}^d\Big( h_e^{3}\|\sigma_b^*\|_e^2+h_T\|
\sigma_{gi}^*\|_e^2 \Big).
\end{split}
\end{equation}
It follows from (\ref{a6}) and (\ref{a9}) that
\begin{equation}\label{a10}
\begin{split}
\3barv\3bar^2\leq C\sum_{e\in{\mathcal E}_h^0}
 h_e^{3}\|\sigma_b\|^2_e+h_e \|
\boldsymbol{\sigma}_{g}\|^2_e = C \|\sigma\|^2_{\Xi_h}.
\end{split}
\end{equation}
Recall that $\sigma^L+\sigma^R=0$. Thus,
\begin{equation}\label{aa10}
\begin{split}
h_e^{-3}\|\jump{v_b}_e\|_e^2+h_e^{-1}\|\jump{\textbf{v}_g}_e
\|_e^2=& h_e^{-3}\|v_b^L-v_b^R\|_e^2+h_e^{-1}
\|\textbf{v}_g^L-\textbf{v}_g^R\|_e^2\\
=& h_e^{-3}\|h_e^{3}\sigma_b^L-h_e^{3}\sigma_b^R\|_e^2+h_e^{-1}
\|h_e \boldsymbol{\sigma}_g^L-h_e \boldsymbol{\sigma}_g ^R\|_e^2\\
=&2h_e^{ 3}\| \sigma_b \|_e^2+2h_e \| \boldsymbol{\sigma}_g \|_e^2.
\end{split}
\end{equation}
Combining (\ref{a5}), (\ref{a10}), (\ref{aa10}) and (\ref{wh0norm})
gives
\begin{equation}\label{a11}
\begin{split}
\sup_{v\in \W_h^0}\frac{b(v,\sigma)}{\|v\|_{\W_h^0}}  \geq &
C\frac{\sum_{e\in {\mathcal
E}_h^0}h_e^{3}\|\sigma_b\|_e^2+h_e\|\boldsymbol{\sigma}_g\|_e^2 } {(
\sum_{e\in {\mathcal E}_h^0}
 h_e^{3}\|\sigma_b\|^2_e+h_e \|
\boldsymbol{\sigma}_{g}\|^2_e)^{\frac{1}{2}}}\\
\geq &C \|\sigma\|_{\Xi_h},
\end{split}
\end{equation}
which completes the proof.
\end{proof}

\section{Error Equations}\label{Section:Error-Eqns}
The goal of this section is to derive an error equation for the
hybridized WG Algorithm (\ref{2.7})-(\ref{2.7q2}). This error
equation shall play an important role in the forthcoming error
analysis.

\begin{lemma}\label{lemma51} \cite{ww} On each element $T\in {\cal T}_h$, let ${\cal
Q}_h$ be the local $L^2$ projection onto $P_{k-2}(T)$. Then, the
$L^2$ projections $Q_h$ and ${\cal Q}_h$ satisfy the following
commutative property:
\begin{equation}
\partial^2_{ij,w}(Q_h w)={\cal Q}_h(\partial^2_{ij} w),\qquad \forall
i,j=1,\ldots,d,
\end{equation}
 for all $w\in H^2(T)$.
\end{lemma}

Let $u$ and $(u_h;\lambda_h) \in \W_h\times \Xi_h$ be the solutions
of (\ref{0.1}) and (\ref{2.7})-(\ref{2.7q2}), respectively. Let
$\lambda=\{\lambda_b,\lambda_g\}$ be given by
$$
\lambda_b=\partial_n(\triangle u),\quad
\boldsymbol{\lambda}_g=-\partial_n(\nabla u) \qquad \mbox{on }
\partial T.
$$
Define error functions by
\begin{equation}\label{eq:error-functions}
e_h=Q_hu-u_h, \quad \epsilon_h=Q_h\lambda-\lambda_h.
\end{equation}

\begin{lemma}
Let $u$ and $(u_h;\lambda_h) \in \W_h\times \Xi_h$ be the solutions
of (\ref{0.1}) and (\ref{2.7})-(\ref{2.7q2}), respectively. Then,
the error functions $e_h$ and $\epsilon_h$ satisfy the following
equations
\begin{eqnarray}\label{4.1}
a_s(e_h,v)+b(v,\epsilon_h)
&=&\ell_u(v),\quad \forall v\in \W_h^0\\
b(\epsilon_h,\rho)&=&0,\qquad\quad\forall
\rho\in\Xi_h,\label{4.1second}
\end{eqnarray}
where
\begin{equation}\label{4.1-ell}
\begin{split}
\ell_u(v)=&\sum_{T\in{\cal
T}_h}\sum_{i,j=1}^d\langle\partial^2_{ij}u-{\cal Q}_h
(\partial^2_{ij}u),(\partial_iv_0-v_{gi})\cdot n_j\rangle_{\partial T}\\
&-\sum_{T\in{\cal
T}_h}\sum_{i,j=1}^d\langle\partial_j(\partial^2_{ij}u-
{\cal Q}_h\partial^2_{ij}u)\cdot n_i,v_0-v_b\rangle_{\partial T}\\
&+s(Q_hu,v).
\end{split}
\end{equation}

\end{lemma}
\begin{proof} The equation (\ref{4.1second}) is obvious from the definition of $\epsilon_h$. It remains to
verify (\ref{4.1}). To this end, from (\ref{2.4new}) we have for any
$\varphi\in P_{k-2}(T)$,
\begin{equation*}
(\varphi,\partial_{ij,w}^2v)_T
=(\partial^2_{ij}v_0,\varphi)_T+\langle v_0-v_b,\partial_j\varphi
\cdot n_i\rangle_{\partial T}-\langle(\partial_i v_0-v_{gi})\cdot
n_j,\varphi\rangle_{\partial T}.
\end{equation*}
By substituting $\varphi$ by $\partial_{ij,w}^2 Q_hu$ and then using
Lemma \ref{lemma51}, we obtain
\begin{equation*}
\begin{split}
&(\partial_{ij,w}^2 Q_hu,\partial_{ij,w}^2v)_T\\
=&(\partial^2_{ij}v_0,{\cal Q}_h(\partial^2_{ij}u))_T+\langle
v_0-v_b,\partial_j({\cal Q}_h
(\partial^2_{ij}u))\cdot n_i\rangle_{\partial T}\\
&-\langle(\partial_i v_0-v_{gi})\cdot n_j,{\cal Q}_h(\partial^2_{ij} u)\rangle_{\partial T}\\
=&(\partial^2_{ij}v_0, \partial^2_{ij}u)_T+\langle
v_0-v_b,\partial_j({\cal Q}_h(\partial^2_{ij}u))
\cdot n_i\rangle_{\partial T}\\
&-\langle(\partial_i v_0-v_{gi})\cdot n_j,{\cal
Q}_h(\partial^2_{ij} u)\rangle_{\partial T},
\end{split}
\end{equation*}
which can be rewritten as
\begin{equation}\label{4.2}
\begin{split}
(\partial^2_{ij}u, \partial^2_{ij}v_0)_T=&
(\partial_{ij,w}^2(Q_hu),\partial_{ij,w}^2v)_T-
\langle v_0-v_b,\partial_j({\cal Q}_h(\partial^2_{ij}u))\cdot n_i\rangle_{\partial T}\\
&+\langle(\partial_i v_0-v_{gi})\cdot n_j,{\cal
Q}_h(\partial^2_{ij} u)\rangle_{\partial T}.
\end{split}
\end{equation}

With $\lambda_b=\partial_n(\triangle u)$ and
$\boldsymbol{\lambda}_g=-\partial_n(\nabla u)$ we have
\begin{equation*}
\begin{split}
b(Q_h\lambda,v)=&\sum_{T\in T_h}\langle
Q_h\lambda,v\rangle_{\partial T} =\sum_{T\in T_h}\langle
\lambda,v\rangle_{\partial
T}\\
=&\sum_{T\in T_h}\langle \boldsymbol{\lambda_g},\textbf{v}_g
\rangle_{\partial T}+
\sum_{T\in T_h}\langle \lambda_b,v_b\rangle_{\partial T}\\
=&\sum_{T\in {\cal T}_h}\sum_{i,j=1}^d \langle-
\partial^2_{ij}u\cdot n_j, v_{gi}\rangle_{\partial T}+ \sum_{T\in
{\cal T}_h}\sum_{i,j=1}^d \langle
\partial_j(\partial^2_{ij}u)\cdot n_i, v_{b}\rangle_{\partial T}.
\end{split}
\end{equation*}
In addition, from the integration by parts,
\begin{equation*}
(\partial^2_{ij}u, \partial^2_{ij}v_0)_T=
((\partial^2_{ij})^2u,v_0)_T+\langle \partial^2_{ij}u,
\partial_i v_0\cdot n_j\rangle_{\partial T}
-\langle \partial_j(\partial^2_{ij}u)\cdot n_i, v_0\rangle_{\partial
T}.
\end{equation*}
Summing over all $T\in {\cal T}_h$ and then using the fact that
$(\triangle^2 u,v_0)=(f,v_0)$, we obtain
\begin{equation*}
\begin{split}
b(Q_h \lambda,v)&+\sum_{T\in {\cal
T}_h}\sum_{i,j=1}^d(\partial^2_{ij}u, \partial^2_{ij}v_0)_T
= (f,v_0)\\
& +\sum_{T\in {\cal T}_h}\sum_{i,j=1}^d\langle \partial^2_{ij}u,(\partial_i v_0-v_{gi})\cdot n_j\rangle_{\partial T}\\
&-\sum_{T\in {\cal T}_h}\sum_{i,j=1}^d\langle
\partial_j(\partial^2_{ij}u)\cdot n_i, v_0-v_b\rangle_{\partial T}.
\end{split}
\end{equation*}
Combining the above equation with (\ref{4.2}) yields
\begin{equation*}
\begin{split}
&b(Q_h \lambda,v)+ (\partial_w^2 Q_hu,\partial_w^2 v)_h\\
=& (f,v_0)+\sum_{T\in{\cal T}_h}\sum_{i,j=1}^d\langle\partial^2_{ij}u-{\cal Q}_h(\partial^2_{ij}u),(\partial_iv_0-v_{gi})\cdot n_j\rangle_{\partial T}\\
&-\sum_{T\in{\cal
T}_h}\sum_{i,j=1}^d\langle\partial_j(\partial^2_{ij}u-{\cal
Q}_h\partial^2_{ij}u)\cdot n_i,v_0-v_b\rangle_{\partial T}.
\end{split}
\end{equation*}
Adding $s(Q_hu,v)$ to both sides of the above equation gives
\begin{equation}\label{4.3}
\begin{split}
 (\partial_w^2 Q_hu,\partial_w ^2v)_h&+s(Q_hu,v)+b(Q_h \lambda,v)
 = (f,v_0)\\
 &+\sum_{T\in{\cal T}_h}\sum_{i,j=1}^d\langle\partial^2_{ij}u-{\cal Q}_h(\partial^2_{ij}u),(\partial_iv_0-v_{gi})\cdot n_j\rangle_{\partial T}\\
&-\sum_{T\in{\cal
T}_h}\sum_{i,j=1}^d\langle\partial_j(\partial^2_{ij}u-{\cal
Q}_h\partial^2_{ij}u)\cdot n_i,v_0-v_b\rangle_{\partial
T}+s(Q_hu,v).
\end{split}
\end{equation}
Subtracting (\ref{2.7}) from (\ref{4.3}) gives the desired equation
(\ref{4.1}). This completes the proof.
\end{proof}

\section{Error Estimates}\label{Section:Error-Estimates} The goal of this section is
to establish some error estimates for the hybridized WG finite
element solution $(u_h;\lambda_h)$ arising from
(\ref{2.7})-(\ref{2.7q2}). The error equations
(\ref{4.1})-(\ref{4.1second}) imply
\begin{eqnarray*}
a_s(Q_hu-u_h,v) + b(v,Q_h\lambda-\lambda_h)&=& \ell_u(v), \quad \forall v\in  \W_h^0,\\
 b(Q_hu-u_h,\rho)&=&0, \ \ \qquad \forall \rho\in \Xi_h,
\end{eqnarray*}
where $\ell_u(v)$ is given by (\ref{4.1-ell}). The above is a saddle
point problem for which the Brezzi's theorem \cite{b1974} can be
applied for an analysis on its stability and solvability. Note that
all the conditions of Brezzi's theorem have been verified in Section
\ref{Section:Stability} (see Lemmas
\ref{lemboundedness}-\ref{lemmainfsup}).

\begin{theorem} \label{theoestimate}
Let $u$ and $(u_h;\lambda_h) \in \W_h\times \Xi_h$ be the solutions
of (\ref{0.1}) and (\ref{2.7})-(\ref{2.7q2}) respectively. Then,
there exists a constant $C$ such that
\begin{equation}\label{esti1}
\|Q_hu-u_h\|_{\W_h^0}+\|Q_h \lambda - \lambda_h\|_{\Xi_h} \leq
Ch^{k-1}\Big(\|u\|_{k+1}+\delta_{k,2}\|u\|_4\Big),
\end{equation}
where $\delta_{i,j}$ is the Kronecker's delta with value $1$ for
$i=j$ and $0$ otherwise.
\end{theorem}

\begin{proof} From the Brezzi's theorem \cite{b1974}, we have
\begin{equation}\label{f}
\|Q_hu-u_h\|_{\W_h^0} + \|Q_h \lambda - \lambda_h\|_{\Xi_h} \leq
C\|\ell_u\|_{{\W_h^0}'}.
\end{equation}
For any $v \in \W_h^0$, it has been shown in \cite{ww} that
$$
|\ell_u(v)| \leq
Ch^{k-1}\Big(\|u\|_{k+1}+\delta_{k,2}\|u\|_4\Big)\3bar v \3bar.
$$
Thus, we have
\begin{equation}\label{f2}
\|\ell_u\|_{{\W_h^0}'} =\sup_{v \in W_h^0 }\frac{\ell_u(v )}{\|v
\|_{\W_h^0}}\leq\sup_{v \in \W_h^0 }\frac{\ell_u(v )}{\3barv
\3bar}\leq Ch^{k-1} \Big(\|u\|_{k+1}+\delta_{k,2}\|u\|_4\Big).
\end{equation}
Substituting (\ref{f2}) into (\ref{f}) yields the desired estimate
(\ref{esti1}), which completes the proof.
\end{proof}

\begin{theorem}
Let $u$ and $ \lambda_h=\{\lambda_{h,b},\boldsymbol{\lambda}_{h,g}\}
\in \Xi_h$ be  the solution  of (\ref{0.1}) and part of the solution
of (\ref{2.7})-(\ref{2.7q2}), respectively. On the set of interior
edges $\E_h^0$, let $\lambda=\{\lambda_b,\boldsymbol{\lambda}_g
 \}$ be given by
$$
\lambda_b= \partial_n(\Delta u),\qquad
 \boldsymbol{\lambda}_g=
 -\partial_n(\nabla u).
$$
Then, the following estimate holds true
\begin{equation}\label{lambdab}
 \|\lambda - \lambda_{h}\|_{\Xi_h}
\leq Ch^{k-1}\big(\|u\|_{k+1}+ \delta_{k,2} \|u\|_{4}\big).
\end{equation}
\end{theorem}

\begin{proof}
From the triangle inequality,
\begin{equation}\label{lambda1}
\|\lambda - \lambda_{h}\|_{\Xi_h} \leq \| \lambda - Q_h\lambda
\|_{\Xi_h} +\| Q_h\lambda - \lambda_{h}\|_{\Xi_h}.
\end{equation}
The second term on the right-hand side of (\ref{lambda1}) can be
handled by (\ref{esti1}). The first term is simply the error between
$\lambda$ and its $L^2$ projection, and can be rewritten as
\begin{equation}\label{lambda2}
\begin{split}
\|\lambda - Q_h\lambda \|_{\Xi_h}^2 = & \sum_{e\in {\mathcal
E}_h^0}h_e^3\|\lambda_b - Q_b\lambda\|_e^2 +  h_e\|{\blambda}_g - Q_b\blambda_g\|_e^2\\
= & \sum_{e\in {\mathcal
E}_h^0}h_e^3\|\partial_n\Delta u - Q_b(\partial_n\Delta u)\|_e^2 +  h_e\|\partial_n\nabla u -
Q_b(\partial_n\nabla u)\|_e^2.\\
\end{split}
\end{equation}
Let $e$ be an edge of the element $T$ and denote by $Q_{k-1}$ the
$L^2$ projection onto $P_{k-1}(T)$. From the trace inequality
(\ref{eq:trace inequality-p=2}), we obtain
\begin{equation}\label{lambda2.100}
\begin{split}
&\|\partial_n\Delta u - Q_b(\partial_n\Delta u)\|_e^2 \\
\leq &\|\partial_n\Delta u - \partial_n(Q_{k-1}\Delta u)\|_e^2\\
\leq &C h^{-1} \|\Delta u - Q_{k-1}\Delta u\|_{1,T}^2 + C h \|\Delta
u - Q_{k-1}\Delta u\|_{2,T}^2\\
\leq &Ch^{2k-5}\|u\|_{k+1,T}^2 + C h\delta_{k,2}\|u\|_{4,T}^2.
\end{split}
\end{equation}
Analogously,
\begin{equation}\label{lambda2.110}
\begin{split}
&\|\partial_n\nabla u - Q_b(\partial_n\nabla u)\|_e^2 \\
\leq &\|\partial_n\nabla u - \partial_n(Q_{k-1}\nabla u)\|_e^2\\
\leq & C h^{-1} \|\nabla u - Q_{k-1}\nabla u\|_{1,T}^2 + C h
\|\nabla
u - Q_{k-1}\nabla u\|_{2,T}^2\\
\leq & Ch^{2k-3}\|u\|_{k+1,T}^2.
\end{split}
\end{equation}
Substituting (\ref{lambda2.100}) and (\ref{lambda2.110}) into
(\ref{lambda2}) yields
\begin{equation}\label{lambda2.800}
\|\lambda - Q_h\lambda \|_{\Xi_h}^2 \leq Ch^{2k-2}(\|u\|_{k+1}^2 + C
h^2\delta_{k,2}\|u\|_{4}^2).
\end{equation}
This completes the proof of the theorem.
\end{proof}

\section{Efficient Implementation via Variable
Reduction}\label{Section:Variable-Reduction}

The degrees of freedom in the WG algorithm (\ref{algori1}) can be
divided into two classes: (1) the interior variables representing
$u_0$, and (2) the interface variables for $\{u_b, \textbf{u}_g\}$.
For the hybridized WG algorithm (\ref{2.7})-(\ref{2.7q2}), more
unknowns must be added to the picture from the Lagrange multiplier
$\lambda_h$. Thus, the size of the discrete system arising from
either (\ref{algori1}) or (\ref{2.7})-(\ref{2.7q2}) is enormously
large.

The goal of this section is to present a Schur complement
formulation for the WG algorithm (\ref{algori1}) based on the
hybridized formulation (\ref{2.7})-(\ref{2.7q2}). The method shall
eliminate all the unknowns associated with $u_0$, and produce a much
reduced system of linear equations involving only the unknowns
representing the interface variables $\{u_b, \textbf{u}_g\}$.

\subsection{Theory of variable reduction} Denote by $\B_h$ the interface finite element
space defined as the restriction of the finite element space $\V_h$
on the set of edges ${\mathcal E}_h$; i.e.,
$$
\B_h=\{\{\mu_b,\boldsymbol{\mu}_g\}:\mu_b\in
P_{k-2}(e),\boldsymbol{\mu}_g\in [P_{k-2}(e)]^d, e\in {\mathcal
E}_h\}.
$$
$\B_h$ is a Hilbert space equipped with the following inner product
$$
\langle \{w_b,\textbf{w}_g\},\{q_b,\textbf{q}_g\}\rangle_{ {\mathcal
E}_h}=\sum_{e\in  {\mathcal E}_h} \langle w_b,q_b\rangle_e+\langle
\textbf{w}_g,\textbf{q}_g\rangle_e ,
\qquad\forall\{w_b,\textbf{w}_g\}, \{q_b,\textbf{q}_g\}\in \B_h.
$$
Denote by $\B_h^0$ the subspace of $\B_h$ consisting of functions
with vanishing boundary value.

We introduce an operator $S_f:\ \B_h\to \B_h^0$ as follows. For any
$\{w_b,\textbf{w}_g\}\in \B_h$, the image
$S_f(\{w_b,\textbf{w}_g\})$ is obtained as follows:

\begin{description}
\item[Step 1.] On each element $T\in {\cal T}_h$, compute $w_0$ in
terms of $\{w_b,\textbf{w}_g\}$ by solving the following local
equations
\begin{equation}\label{6.1}
a_{s,T}(w_h,v)=(f,v_0)_T, \qquad \forall v=\{v_0,0,\textbf{0}\}\in
\W_k(T),
\end{equation}
where $w_h=\{w_0,w_b,\textbf{w}_g\}\in \W_k(T)$. We denote the
solution by $w_0=D_f(\{w_b,\textbf{w}_g\})$.

\item[\bf Step 2.]  Compute $\zeta_{h,T}\in \Lambda_k(\partial T)$ on each element $T\in {\cal T}_h$
such that
\begin{equation}\label{6.2}
b_T(v,\zeta_{h,T})=a_{s,T}(w_h,v), \quad \forall
v=\{0,v_b,\textbf{v}_g\}\in \W_k(T).
\end{equation}
This provides a function $\zeta_h\in \Lambda_h$. Denote $\zeta_h$ by
$\zeta_h=L_f(\{w_b,\textbf{w}_g\})$.

\item[\bf Step 3.] Set $S_f(\{w_b,\textbf{w}_g\})$ as
the similarity of $\zeta_h$ on interior edges and zero on boundary
edges; i.e.,
\begin{equation}\label{6.3}
\begin{split}
S_f(\{w_b,\textbf{w}_g\}) = \left\{\begin{array}{ll}
{\zeta_h}_{L}+{\zeta_h}_{R}, &
\mbox{on\ } e\in{\mathcal E}_h^0, \\
0, & \mbox{on\ } e\subset \partial\Omega.
\end{array}\right.
\end{split}
\end{equation}
\end{description}

By adding the two equations (\ref{6.1}) and (\ref{6.2}), we obtain
the following identity
\begin{equation}\label{6.4}
b_T(v,\zeta_{h,T})=a_{s,T}(w_h,v)-(f,v_0)_T, \quad \forall
v=\{v_0,v_b,\textbf{v}_g\}\in \W_k(T).
\end{equation}
From the superposition principle one has the following result.

\begin{lemma} For any $\{w_b,\textbf{w}_g\}\in \B_h$, we have
\begin{equation}\label{6.5}
S_f(\{w_b,\textbf{w}_g\})=S_0(\{w_b,\textbf{w}_g\})+S_f(\{0,\textbf{0}\}).
\end{equation}
Here $S_0$ is the operator corresponding to the case of $f=0$.
\end{lemma}

It is clear that $S_0$ is a linear map from $\B_h$ into $\B^0_h$.
Moreover, the following result can be verified for $S_0$.

\begin{theorem}\label{THM:theorem6.2} For any $\{w_b,\textbf{w}_g\},
\{q_b,\textbf{q}_g\}\in \B_h^0$, we have
\begin{equation}\label{6.6}
\langle
S_0(\{w_b,\textbf{w}_g\}),\{q_b,\textbf{q}_g\}\rangle_{{\mathcal
E}_h^0}=a_s(w_h,q_h),
\end{equation}
where $w_h=\{D_0(\{w_b,\textbf{w}_g\}),w_b,\textbf{w}_g\}$ and
$q_h=\{D_0(\{q_b,\textbf{q}_g\}), q_b,\textbf{q}_g\}$. In other
words, the linear map $S_0$, when restricted to the subspace
$\B_h^0$, is symmetric and positive definite.
\end{theorem}

\begin{proof} For any
$\{w_b,\textbf{w}_g\}$, $\{q_b,\textbf{q}_g\}\in \B_h^0$, let
\begin{eqnarray*}
&w_h=\{D_0(\{w_b,\textbf{w}_g\}),w_b,\textbf{w}_g\},\qquad
&\zeta_h=L_0(\{w_b,\textbf{w}_g\}),\\
&q_h=\{D_0(\{q_b,\textbf{q}_g\}), q_b,\textbf{q}_g\}, \qquad
&\eta_h=L_0(\{q_b,\textbf{q}_g\}).
\end{eqnarray*}
Using (\ref{6.4}) with $f=0$ we arrive at
\begin{equation*}
\begin{split}
\langle
S_0(\{w_b,\textbf{w}_g\}),\{q_b,\textbf{q}_g\}\rangle_{{\mathcal
E}_h^0}=&\sum_{e\in {\mathcal E}_h^0}\langle
\similarity{ \zeta_h}_e,\{q_b,\textbf{q}_g\}\rangle_e\\
=& \sum_{T\in {\cal T}_h}\langle\zeta_{h,T},\{q_b,\textbf{q}_g\}\rangle_{\partial T}\\
=&\sum_{T\in {\cal T}_h} b_T(q_h,\zeta_{h,T})\\
=&\sum_{T\in {\cal T}_h} a_{s,T}(w_h,q_h),
\end{split}
\end{equation*}
 which
completes the proof.
\end{proof}

\begin{lemma}\label{lem7.2} Let
$(u_h;\lambda_h)=(\{u_0,u_b,\textbf{u}_g\};\lambda_h)\in \W_h\times
\Xi_h$ be the unique solution of the hybridized WG algorithm
(\ref{2.7})-(\ref{2.7q2}). Then, $u_h\in \V_h$ and
$\{u_b,\textbf{u}_g\}$ is well defined  in the space $\B_h$.
Moreover, they satisfy the following equation
\begin{equation}\label{6.7}
S_f(\{u_b,\textbf{u}_g\})=\{0,\textbf{0}\}.
\end{equation}
\end{lemma}

\begin{proof} Since $(u_h;\lambda_h)$ is the unique solution
 of the hybridized WG algorithm (\ref{2.7})-(\ref{2.7q2}), then we have from
 Lemma \ref{lem3.1} that $\jump{u_h}_e=0$ on each
interior edge or flat face $e\in {\cal E}_h^0$. Furthermore, on each
boundary edge, we have $u_b=Q_b\xi$, $\textbf{u}_g\cdot
\textbf{n}=Q_{b}\nu$, $\textbf{u}_g\cdot
\boldsymbol{\tau}=Q_{b}(\nabla\xi\cdot\btau)$. Thus, $u_h\in \V_h$
and its restriction on ${\mathcal E}_h$ is a well defined function
in the space $\B_h$.

Now in (\ref{2.7}), choose $v=\{v_0,0,\textbf{0}\}\in \W_k(T)$ on
$T$ and zero elsewhere. Then,
$$
a_{s,T}(u_h,v)=(f,v_0)_T,\  \forall v=\{v_0,0,\textbf{0}\}\in
\W_k(T).
$$
This implies that $u_h$ satisfies the local equation (\ref{6.1}).

Next in (\ref{2.7}), choose $v=\{0,v_b,\textbf{v}_g\}\in W_k(T)$ on
$T$ and zero elsewhere. Then,
$$
b_T(\lambda_{h,T},v)=a_{s,T}(u_h,v),\  \forall
v=\{0,v_b,\textbf{v}_g\}\in \W_k(T),
$$
where $\lambda_{h,T}$ is the restriction of $\lambda_h$ on the
boundary of $T$. This means that $\lambda_h$ satisfies (\ref{6.2}).

From the definition of the operator $S_f$, we have on interior edges
$$
S_f(\{u_b,\textbf{u}_g\})=\similarity{\lambda_h}.
$$
$\lambda_h\in \Xi_h$ implies $ \similarity{ \lambda_h}
=\{0,\textbf{0}\}$, and hence
$S_f(\{u_b,\textbf{u}_g\})=\{0,\textbf{0}\}$. This completes the
proof of the theorem.
\end{proof}

\begin{lemma}\label{lem7.3} Let
$\{\overline{u}_b,\overline{\textbf{u}}_g\}\in \B_h$ satisfy $
\overline{u}_b=Q_b\xi$ and $\overline{\textbf{u}}_g\cdot
\textbf{n}=Q_{b}\nu$, $\overline{\textbf{u}}_g\cdot
\boldsymbol{\tau}=Q_{b}(\nabla\xi\cdot\btau)$ on $\partial\Omega$
and the following operator equation
\begin{equation}\label{6.10}
S_f(\{\overline{u}_b,\overline{\textbf{u}}_g\})=\{0,\textbf{0}\}.
\end{equation}
Then, $\overline{u}_h=\{\overline{u}_0,\overline{u}_b,
\overline{\textbf{u}}_g\}\in \V_h$ is the solution of the WG
algorithm (\ref{algori1}). Here $\overline{u}_0$ is the solution of
the following local problems on each element $T\in {\cal T}_h$,
\begin{equation}\label{6.11}
a_{s,T}(\overline{u}_h,v)=(f,v_0)_T,\  \forall
v=\{v_0,0,\textbf{0}\}\in \W_k(T).
\end{equation}
\end{lemma}

\begin{proof} Let $\{\overline{u}_b,\overline{\textbf{u}}_g\}\in \B_h$
satisfy the operator equation (\ref{6.10}) and the said boundary
condition. Let $\overline{u}_0$ be given by the local equations
(\ref{6.11}). Now on each element $T$, we compute
$\overline{\lambda}_{h,T}\in \Lambda_k(\partial T)$ by solving the
local problem
\begin{equation}\label{6.12}
b_T(v,\overline{\lambda}_{h,T})=a_{s,T}(\overline{u}_h,v), \quad
\forall v=\{0,v_b,\textbf{v}_g\}\in \W_k(T).
\end{equation}
This defines a function $\overline{\lambda}_{h}\in \Lambda_h$ given
by $\overline{\lambda}|_{\partial T}=\overline{\lambda}_{h,T}$ with
modification $\overline{\lambda}|_{\partial\Omega}=0$. From the
definition of the operator $S_f$, on each interior edge
$e\in\E_h^0$, we have
$$
S_f(\{\overline{u}_b,\overline{\textbf{u}}_g\})=
\similarity{\overline{\lambda}_h},
$$
which, together with (\ref{6.10}) leads to
\begin{equation}\label{6.13}
\similarity{\overline{\lambda}_h} =\{0,\textbf{0}\}
\end{equation}
on each interior edge. Thus, $\overline{\lambda}_h\in \Xi_h$.

Subtracting (\ref{6.12}) from (\ref{6.11}) gives
$$
a_{s,T}(\overline{u}_h,v)-b_T(v,\overline{\lambda}_{h,T})=(f,v_0)_T,
\quad\forall v=\{v_0,v_b,\textbf{v}_g\}\in \W_k(T).
$$
Summing up the above equation over all elements $T\in {\cal T}_h$
gives
\begin{equation}\label{night-late:001}
a_s(\overline{u}_h,v)-b(v,\overline{\lambda}_{h})=(f,v_0), \quad
\forall v=\{v_0,v_b,\textbf{v}_g\}\in \W_h^0.
\end{equation}
Note that the above equation holds true only for test functions $v$
with vanishing boundary value since $\lambda_h$ was modified from
$\lambda_{h,T}$ on the boundary of the domain.

For any $\sigma$ in the finite element space $\Xi_h$, we have from
(\ref{EQ:b-form}) that
\begin{equation}\label{night-late:002}
b(\overline{u}_h,\sigma)=\sum_{e\subset {\mathcal E}_h^0}\langle
\jump{\overline{u}_h}_e,\sigma_L\rangle_e=0.
\end{equation}
The equations (\ref{night-late:001}) and (\ref{night-late:002})
indicate that $(\overline{u}_h; \overline{\lambda}_{h})$ is a
solution to the hybridized WG scheme (\ref{2.7})-(\ref{2.7q2}).
Recall that on the boundary $\partial\Omega$, we have $
\overline{u}_b=Q_b\xi$ and $\overline{\textbf{u}}_g\cdot
\textbf{n}=Q_{b}\nu$, $\overline{\textbf{u}}_g\cdot
\boldsymbol{\tau}=Q_{b}(\nabla\xi\cdot\btau)$. Thus, using Theorem
\ref{lem3.1} we see that $\overline{u}_h$ is the WG solution defined
by the formulation (\ref{algori1}). This completes the proof of the
theorem.
\end{proof}

The results developed in Lemmas \ref{lem7.2} -\ref{lem7.3} can be
summarized as follows.

\begin{theorem} Let $\{\overline{u}_b,\overline{\textbf{u}}_g\}\in \B_h$
be any function such that $ \overline{u}_b=Q_b\xi$ and
$\overline{\textbf{u}}_g\cdot \textbf{n}=Q_{b}\nu$,
$\overline{\textbf{u}}_g\cdot
\boldsymbol{\tau}=Q_{b}(\nabla\xi\cdot\btau)$ on $\partial\Omega$.
Define $\overline{u}_0$ as the solution of (\ref{6.11}). Then,
$\overline{u}_h=\{\overline{u}_0,
\overline{u}_b,\overline{\textbf{u}}_g\}$ is the solution of
(\ref{algori1}) if any only if
$\{\overline{u}_b,\overline{\textbf{u}}_g\}$ satisfies the following
operator equation
\begin{equation}\label{6.15}
S_f(\{\overline{u}_b,\overline{\textbf{u}}_g\})=\{0,\textbf{0}\}.
\end{equation}
\end{theorem}

\subsection{Computational algorithm with reduced variables}

From (\ref{6.5}), the operator equation (\ref{6.15}) can be
rewritten as
\begin{equation}\label{6.16}
S_0(\{\overline{u}_b,\overline{\textbf{u}}_g\})=-S_f(\{0,\textbf{0}\}).
\end{equation}
Let $\{G_b,\textbf{G}_g\}\in \B_h$ be a finite element function
satisfying $ G_b=Q_b\xi$,  $ \textbf{G}_g\cdot \textbf{n}=Q_{b}\nu$
and $ \textbf{G}_g\cdot
\boldsymbol{\tau}=Q_{b}(\nabla\xi\cdot\btau)$ on $\partial\Omega$
and zero elsewhere. It follows from the linearity of $S_0$ that
$$
 S_0(\{\overline{u}_b,\overline{\textbf{u}}_g\})=
 S_0(\{\overline{u}_b,\overline{\textbf{u}}_g\}-\{G_b,\textbf{G}_g\})+S_0(\{G_b,\textbf{G}_g\}).
$$
Substituting the above into (\ref{6.16}) yields
$$
 S_0(\{\overline{u}_b,\overline{\textbf{u}}_g\}-\{G_b,\textbf{G}_g\})=
 -S_f(\{0,\textbf{0}\})-S_0(\{G_b,\textbf{G}_g\}).
$$
Note that the function
$\{p_b,\textbf{p}_g\}=\{\overline{u}_b,\overline{\textbf{u}}_g\}-\{G_b,\textbf{G}_g\}$
has vanishing boundary value. By setting $\{r_b,\textbf{r}_g\}=
-S_f(\{0,\textbf{0}\})-S_0(\{G_b,\textbf{G}_g\})$, we have
\begin{equation}\label{6.17}
S_0(\{p_b,\textbf{p}_g\})=\{r_b,\textbf{r}_g\}.
\end{equation}
The reduced system of linear equations (\ref{6.17}) is actually a
Schur complement formulation for the WG algorithm (\ref{algori1}).
Note that (\ref{6.17}) involves only the variables representing the
value of the function on $\E_h^0$. This is clearly a significant
reduction on the size of the linear system that has to be solved in
the WG finite element method.

\begin{variable-reduction}\label{Alg:variable-reduction} The solution
$u_h=\{u_0,u_b,\textbf{u}_g\}$ to the WG algorithm (\ref{algori1})
can be obtained step-by-step as follows:

\begin{itemize}
\item[(1)] On each element $T$, compute
$$
r_h = -S_f(\{0,\textbf{0}\})-S_0(\{G_b,\textbf{G}_g\}).
$$
This task requires the inversion of local stiffness matrices and can
be accomplished in parallel. The computational complexity is linear
with respect to the number of unknowns.

\item[(2)] Compute $\{p_b,\textbf{p}_g\}\in \B_h^0$ by solving the system of
linear equations (\ref{6.17}). This step requires an efficient
linear solver.

\item[(3)] Compute
$\{u_b,\textbf{u}_g\}=\{p_b,\textbf{p}_g\}+\{G_b,\textbf{G}_g\}$ to
get the solution on element boundaries. Then, on each element $T$,
compute $u_0=D_f(\{u_b,\textbf{u}_g\})$ by solving the local problem
(\ref{6.1}). This task can be accomplished in parallel, and the
computational complexity is proportional to the number of unknowns.
\end{itemize}

\end{variable-reduction}

Step (2) in the {\sc Variable Reduction Algorithm
}\ref{Alg:variable-reduction} is the only computation-extensive part
of the implementation. Note that, due to Theorem
\ref{THM:theorem6.2}, the reduced system (\ref{6.17}) is symmetric
and positive definite. Preconditioning techniques should be applied
for an efficient solving of (\ref{6.17}). This is left to interested
readers for an investigation.

\end{document}